\documentclass[12pt]{article}
\usepackage{amssymb}
\usepackage{amsmath}
\usepackage{amsthm}
\usepackage{color}
\usepackage{geometry}		
\usepackage{graphicx}

\DeclareFontFamily{OT1}{pzc}{}
\DeclareFontShape{OT1}{pzc}{m}{it}{<-> s * [1.10] pzcmi7t}{}
\DeclareMathAlphabet{\mathpzc}{OT1}{pzc}{m}{it}

\let\originalleft\left
\let\originalright\right
\renewcommand{\left}{\mathopen{}\mathclose\bgroup\originalleft}
\renewcommand{\right}{\aftergroup\egroup\originalright}

\newcommand{\myStep}[2]{{\bf Step #1} --- #2\\}

\newcommand{\bM}{{\bf M}}
\newcommand{\bO}{{\bf 0}}
\newcommand{\co}{\mathpzc{o}}
\newcommand{\ee}{\varepsilon}
\newcommand{\rD}{{\rm D}}
\newcommand{\re}{{\rm e}}
\newcommand{\sect}{section~}

\newtheorem{theorem}{Theorem}[section]
\newtheorem{lemma}[theorem]{Lemma}

\theoremstyle{definition}
\newtheorem{definition}{Definition}[section]

\begin{document}

\title{
Inclusion of higher-order terms in the border-collision normal form: persistence of chaos and applications to power converters.}
\author{
D.J.W.~Simpson$^{\dagger}$ and P.A.~Glendinning$^{\ddagger}$\\
\small $^{\dagger}$School of Fundamental Sciences, Massey University, Palmerston North, New Zealand\\
\small $^{\ddagger}$Department of Mathematics, University of Manchester, UK
}
\maketitle

\begin{abstract}

The dynamics near a border-collision bifurcation are approximated to leading order by a continuous, piecewise-linear map.
The purpose of this paper is to consider the higher-order terms that are neglected when forming this approximation.
For two-dimensional maps we establish conditions under which a chaotic attractor created in a border-collision bifurcation
persists for an open interval of parameters beyond the bifurcation.
We apply the results to a prototypical power converter model to prove the model exhibits robust chaos.

\end{abstract}

\section{Introduction}
\label{sec:intro}
\setcounter{equation}{0}

While most engineering systems are designed to operate outside chaotic parameter regimes,
in some situations the presence of chaos is advantageous.
Examples include optical resonators for which
additional energy can be stored when photons follow chaotic trajectories
as without periodic motion they become trapped for longer times \cite{LiDi13}.
In mechanical energy harvesters the presence of chaos allows
high-energy modes of operation to be stabilised with only a small control force \cite{KuAl16}.
Also, power converters, when run chaotically, have the advantage of
increased electromagnetic compatibility due to broad spectral characteristics \cite{DeHa96}.
Regardless of whether or not chaos is desired, it is extremely helpful to understand when and why it occurs.

Power converters, and many other engineering systems,
function by switching between different modes of operation.
There is a growing understanding of the creation of chaos in such systems through use of the border-collision normal form (BCNF).
This is a piecewise-linear (or more precisely, piecewise-affine) map that models the
dynamics near parameter values at which
a fixed point intersects a boundary between two modes of operation \cite{NuYo92,DiBu08}.
The piecewise-linear nature of the BCNF makes it possible to prove results about chaotic attractors and their persistence,
a phenomenon called robust chaos \cite{BaYo98,BaGr99,Gl17}.

The BCNF is obtained using coordinate transformations to make the lowest order terms of the map as simple as possible.
The map is then truncated, ignoring all terms that are nonlinear with respect to variables and parameters.
As such, it is by no means clear that results for chaotic attractors in the BCNF
carry over to the full nonlinear models originally being considered.
In this paper we prove a persistence result showing that, under certain conditions,
the existence of a chaotic attractor in the BCNF implies the existence of a chaotic attractor in the corresponding full model
regardless of the nature of the nonlinear terms that have been neglected to form the BCNF.
This shows that chaotic attractors created in border-collision bifurcations
typically persist for an open interval of parameters beyond the bifurcation
and justifies our use of the BCNF for determining when chaotic attractors are created.

In earlier work \cite{GlSi21b} we described a computational method for determining
when the two-dimensional BCNF is chaotic by establishing the existence of a trapping region in phase space
and a contracting-invariant expanding cone in tangent space.
The trapping region guarantees the existence of an attractor,
while the cone ensures it has a positive Lyapunov exponent.
In this paper we use the robustness of these objects to demonstrate persistence with respect to higher-order terms.
The most difficult technical issue we have to overcome
is in showing that the higher-order terms do not cause orbits of the map
to accumulate new symbolic itineraries that cannot be handled by the cone.

The remainder of the paper is arranged as follows.
In \sect\ref{sec:bcnf} we describe the two-dimensional BCNF and its relation to nonlinear models.
We then state our persistence result in \sect\ref{sec:main}.
In \sect\ref{sec:leftright} we provide a detailed description of the geometric structure of the phase space of the BCNF
that we use in \sect\ref{sec:proof} to prove the persistence result.
The result, together with computational methods of \cite{GlSi21b},
are then applied to a model of power converters with pulse-width modulated control, \sect\ref{sec:application}.
We believe this verifies, for the first time, the long-standing belief from numerical simulations that chaos
occurs robustly in these types of systems.
Finally \sect\ref{sec:conc} contains concluding remarks.

\section{Border-collision bifurcations and the border-collision normal form}
\label{sec:bcnf}
\setcounter{equation}{0}

The study of border-collision bifurcations has a long history dating
back to at least Feigin and coworkers in the context of relay control \cite{BrNe63,Fe70z}.
The two-dimensional BCNF was introduced
by Nusse and Yorke in \cite{NuYo92}
motivated by observations of anomalous behaviour in piecewise-linear economics models \cite{HoNu91}.
The BCNF has since been shown to display a remarkably rich array of dynamics,
such as robust chaos \cite{BaYo98},
multi-stability \cite{DuNu99},
multi-dimensional attractors \cite{Gl16e},
and resonance regions with sausage-string structures \cite{SiMe09,Si17c}.
The BCNF is relevant for describing a wide-range of physical phenomena,
another example being mechanical systems with stick-slip friction \cite{DiKo03};
see \cite{Si16} for a recent review.

Let $y \mapsto f(y;\mu)$ be a continuous, piecewise-smooth map with variable $y = (y_1,y_2) \in \mathbb{R}^2$
and parameter $\mu \in \mathbb{R}$.
We are interested in the dynamics local to a border-collision bifurcation
where a fixed point of $f$ collides with a switching manifold as $\mu$ is varied.
Assuming the bifurcation occurs at single smooth switching manifold,
only two pieces of the map are relevant to the local dynamics.
By choosing coordinates so that the switching manifold is $y_1 = 0$, we can assume $f$ has the form
\begin{equation}
f(y;\mu) = \begin{cases}
f_L(y;\mu), & y_1 \le 0, \\
f_R(y;\mu), & y_1 \ge 0,
\end{cases}
\label{eq:f}
\end{equation}
where $f_L$ and $f_R$ are $C^1$.

Suppose the border-collision bifurcation occurs at $y = (0,0) = \bO$ when $\mu = 0$.
By continuity, $\bO$ is a fixed point of both $f_L$ and $f_R$ when $\mu = 0$:
\begin{equation}
f_L(\bO;0) = f_R(\bO;0) = \bO.
\label{eq:bcbCond}
\end{equation}
From the map $f$ we can extract the four key values
\begin{equation}
\begin{split}
\tau_L &= {\rm trace} \big( \rD f_L(\bO;0) \big), \\
\delta_L &= {\rm det} \big( \rD f_L(\bO;0) \big), \\
\tau_R &= {\rm trace} \big( \rD f_R(\bO;0) \big), \\
\delta_R &= {\rm det} \big( \rD f_R(\bO;0) \big),
\end{split}
\label{eq:tLdLtRdR}
\end{equation}
which determine the eigenvalues
associated with $\bO$
for the two smooth components of \eqref{eq:f}.
We can then use these values to form the piecewise-linear map
\begin{equation}
g(x) = \begin{cases}	
\begin{bmatrix} \tau_L & 1 \\ -\delta_L & 0 \end{bmatrix} x + \begin{bmatrix} 1 \\ 0 \end{bmatrix}, & x_1 \le 0, \\[4.5mm]
\begin{bmatrix} \tau_R & 1 \\ -\delta_R & 0 \end{bmatrix} x + \begin{bmatrix} 1 \\ 0 \end{bmatrix}, & x_1 \ge 0.
\end{cases}
\label{eq:bcnf}
\end{equation}
This is the two-dimensional BCNF,
except often $\mu$-dependence is retained in the constant term.
In the remainder of this section we explain how \eqref{eq:bcnf} can be used
to describe the local dynamics of \eqref{eq:f} near the border-collision bifurcation.

We first write the components of \eqref{eq:f} as
\begin{equation}
\begin{split}
f_L(y;\mu) &=
\begin{bmatrix}
a^L_{11} & a_{12} \\
a^L_{21} & a_{22}
\end{bmatrix} y + \begin{bmatrix} b_1 \\ b_2 \end{bmatrix} \mu + \co \left( \| y \| + |\mu| \right), \\
f_R(y;\mu) &=
\begin{bmatrix}
a^R_{11} & a_{12} \\
a^R_{21} & a_{22}
\end{bmatrix} y + \begin{bmatrix} b_1 \\ b_2 \end{bmatrix} \mu + \co \left( \| y \| + |\mu| \right),
\end{split}
\label{eq:fLfR}
\end{equation}
where the two expressions share several coefficients
due to the assumed continuity of \eqref{eq:f} on $y_1 = 0$.
By dropping the higher-order terms we obtain the piecewise-linear map
\begin{equation}
h(y;\mu) = \begin{cases}
\begin{bmatrix} a^L_{11} & a_{12} \\ a^L_{21} & a_{22} \end{bmatrix} y +
\begin{bmatrix} b_1 \\ b_2 \end{bmatrix} \mu, & y_1 \le 0, \\[4.5mm]
\begin{bmatrix} a^R_{11} & a_{12} \\ a^R_{21} & a_{22} \end{bmatrix} y +
\begin{bmatrix} b_1 \\ b_2 \end{bmatrix} \mu, & y_1 \ge 0.
\end{cases}
\label{eq:fpwl}
\end{equation}
This approximates \eqref{eq:f} in the following sense:
for any $\ee > 0$ there exists $\delta > 0$ such that
if $\| y \| + |\mu| < \delta$ then
$\big\| f(y;\mu) - h(y;\mu) \big\| < \ee \left( \| y \| + |\mu| \right)$.

The piecewise-linear map \eqref{eq:fpwl} satisfies the identity
$h(\alpha y;\alpha \mu) = \alpha h(y;\mu)$ for any $\alpha > 0$.
For this reason the magnitude of $\mu$ only affects the spatial scale of the dynamics of \eqref{eq:fpwl}:
if $\Lambda \subset \mathbb{R}^2$ is an invariant set of $h$ for some $\mu_0$,
then $\alpha \Lambda$ is an invariant set of $h$ with $\mu = \alpha \mu_0$ for all $\alpha > 0$.
In view of this scaling property, we can convert \eqref{eq:fpwl} to \eqref{eq:bcnf} for any $\mu > 0$.
This is achieved via the change of variables
\begin{equation}
x = \frac{1}{\gamma \mu} \begin{bmatrix} 1 & 0 \\ -a_{22} & a_{12} \end{bmatrix} y
+ \frac{1}{\gamma} \begin{bmatrix} 0 \\ a_{22} b_1 - a_{12} b_2 \end{bmatrix},
\label{eq:ytox}
\end{equation}
where
\begin{equation}
\gamma = (1 - a_{22}) b_1 + a_{12} b_2 \,,
\label{eq:gamma}
\end{equation}
and is valid assuming
\begin{align}
a_{12} &\ne 0, \label{eq:transCond1} \\
\gamma &> 0. \label{eq:transCond2}
\end{align}
The condition $a_{12} \ne 0$ ensures \eqref{eq:ytox} is invertible
(if $a_{12} = 0$ then \eqref{eq:fpwl} can be partly decoupled \cite{Si16}).
We require $\gamma \ne 0$ so that $\mu$ unfolds the border-collision bifurcation in a generic fashion,
while $\gamma > 0$ ensures the left and right components of \eqref{eq:fpwl}
transform to their respective components in \eqref{eq:bcnf}.
The case $\gamma < 0$ can be accommodated by simply redefining $\mu$ as $-\mu$.

The transformation \eqref{eq:ytox} performs a similarity transform to the Jacobian matrices of the components of the map,
thus it preserves their traces and determinants.
For this reason the values \eqref{eq:tLdLtRdR} were used to construct \eqref{eq:bcnf} ---
notice how $\tau_L$, $\delta_L$, $\tau_R$, and $\delta_R$
are the traces and determinants of the Jacobian matrices of the two components of \eqref{eq:bcnf}.

In summary, for any map of the form \eqref{eq:f} that has a border-collision bifurcation at $\mu = 0$,
we can use the values \eqref{eq:tLdLtRdR} to form the BCNF \eqref{eq:bcnf}.
Then, if the conditions \eqref{eq:transCond1}--\eqref{eq:transCond2} are satisfied,
the BCNF is affinely conjugate to the piecewise-linear approximation to \eqref{eq:f} for any $\mu > 0$.
Intuitively this approximation should be reasonable for sufficiently small values of $\mu$.
Our persistence result in the next section gives conditions under which the approximation can indeed be justified.

\section{Persistence of chaotic attractors}
\label{sec:main}
\setcounter{equation}{0}

Our main result, Theorem \ref{th:main} below,
links the existence of chaotic attractors of the piecewise-linear BCNF \eqref{eq:bcnf}
to attractors of the nonlinear map \eqref{eq:f} for small $\mu > 0$.
To state the result we need some preliminary definitions and conditions.

Let
\begin{align*}
g_L(x) &= A_L x + \begin{bmatrix} 1 \\ 0 \end{bmatrix}, &
g_R(x) &= A_R x + \begin{bmatrix} 1 \\ 0 \end{bmatrix},
\end{align*}
denote the left and right components of the BCNF, $g$, where
\begin{align*}
A_L &= \begin{bmatrix} \tau_L & 1 \\ -\delta_L & 0 \end{bmatrix}, &
A_R &= \begin{bmatrix} \tau_R & 1 \\ -\delta_R & 0 \end{bmatrix}.
\end{align*}
Theorem \ref{th:main} requires the assumptions
\begin{align}
\delta_L &> 0, \label{eq:paramCond1} \\
\delta_R &> \frac{\tau_R^2}{4}, \label{eq:paramCond2} \\
g^{-i}(\bO)_2 &< 0, \quad \text{for all $i \ge 1$}, \label{eq:paramCond3}
\end{align}
where the subscript in \eqref{eq:paramCond3} indicates we are looking at the second component of the vector $g^{-i}(\bO)$.
Conditions \eqref{eq:paramCond1} and \eqref{eq:paramCond2} imply that $g$ is a homeomorphism,
so $g^{-1}$ exists and \eqref{eq:paramCond3} is well-defined.
Condition \eqref{eq:paramCond3} ensures the backwards orbit of the origin
does not enter the closed upper half-plane $\left\{ x \in \mathbb{R}^2 \,\big|\, x_2 \ge 0 \right\}$.
In fact, the backwards orbit is constrained to the fourth quadrant of $\mathbb{R}^2$
so is governed purely by $g_R$ and simply converges to the
unique fixed point of $g_R$ (a repelling focus), see Fig.~\ref{fig:zm_origin}.
Together conditions \eqref{eq:paramCond1}--\eqref{eq:paramCond3} allow us to divide
phase space by the number of iterations required to cross the switching manifold
\begin{equation}
\Sigma = \left\{ x \in \mathbb{R}^2 \,\big|\, x_1 = 0 \right\},
\nonumber
\end{equation}
and this is achieved in \sect\ref{sec:leftright}.
Having a precise understanding of this division is critical to our proof of Theorem~\ref{th:main}
presented in \sect\ref{sec:proof}.

\begin{figure}[b!]
\begin{center}
\includegraphics[height=4.5cm]{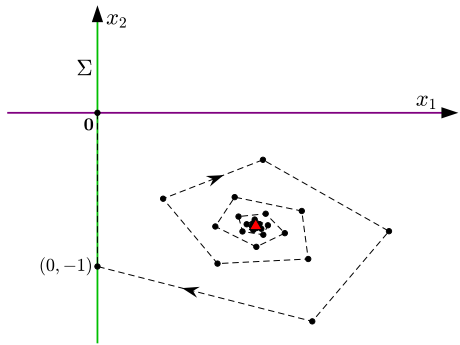}
\caption{
A sketch of the backwards orbit of the origin for the BCNF \eqref{eq:bcnf}
satisfying \eqref{eq:paramCond1}--\eqref{eq:paramCond3}.
The backwards orbit converges to the fixed point of $g_R$ (red triangle).
\label{fig:zm_origin}
}
\end{center}
\end{figure}

To motivate the following definitions,
consider the forward orbit of a point $x \in \mathbb{R}^2$ under $g$.
Suppose it maps under $g_L$ $p$ times, then under $g_R$ $q$ times.
That is, $g^{p+q}(x) = g_R^q \left( g_L^p(x) \right)$,
hence $\rD g^{p+q}(x) = A_R^q A_L^p$,
assuming no iterates lie on $\Sigma$ where $g$ is non-differentiable.
For orbits that go back and forth across $\Sigma$ without landing on $\Sigma$,
which includes almost all orbits in the chaotic attractors we wish to analyse,
derivatives of $g^n(x)$ for large $n$ can be expressed as products of matrices of the form $A_R^q A_L^p$.
Consequently we can estimate Lyapunov exponents based on bounds for the values of $p$ and $q$.

Let
\begin{align*}
\Pi_L &= \left\{ x \in \mathbb{R}^2 \,\big|\, x_1 \le 0 \right\}, \\
\Pi_R &= \left\{ x \in \mathbb{R}^2 \,\big|\, x_1 \ge 0 \right\},
\end{align*}
denote the closed left and right half-planes.

\begin{definition}
Given $x \in \mathbb{R}^2$, let $\chi_L(x)$ be the smallest $p \ge 1$ for which
$g^p(x) \notin \Pi_L$ and let $\chi_R(x)$ be the smallest $q \ge 1$ for which
$g^q(x) \notin \Pi_R$, if such $p$ and $q$ exist.
\label{df:chiLR}
\end{definition}

Now define the regions
\begin{align}
\Phi_L &= \left\{ x \in \mathbb{R}^2 \,\big|\, x_1 < 0,\, x_2 \le 0 \right\}, \label{eq:PhiL} \\
\Phi_R &= \left\{ x \in \mathbb{R}^2 \,\big|\, x_1 > 0,\, x_2 \ge 0 \right\}. \label{eq:PhiR}
\end{align}
The following result shows
that when an orbit crosses $\Sigma$
from right to left, it must arrive at a point in $\Phi_L$.
Moreover, $\Phi_L$ is exactly the set of all such points.
The set $\Phi_R$ admits the same characterisation for orbits crossing $\Sigma$ from left to right.
The result follows immediately from the observation that $g^{-1}(x)_1$ and $x_2$ have opposite signs.

\begin{lemma}
Suppose $\delta_L > 0$ and $\delta_R > 0$.
Then
\begin{align*}
\Phi_L &= \left\{ x \in \mathbb{R}^2 \setminus \Pi_R \,\middle|\, g^{-1}(x) \in \Pi_R \right\}, \\
\Phi_R &= \left\{ x \in \mathbb{R}^2 \setminus \Pi_L \,\middle|\, g^{-1}(x) \in \Pi_L \right\}.
\end{align*}
\label{le:crossingSigma}
\end{lemma}

Now given a set $\Omega \subset \mathbb{R}^2$, suppose there exist
numbers $1 \le p_{\rm min} \le p_{\rm max}$ and $1 \le q_{\rm min} \le q_{\rm max}$ such that
\begin{align}
\chi_L(x) &\ge p_{\rm min} \,, \qquad \text{for all}~x \in \Omega \cap \Phi_L \,, \label{eq:pMinCond} \\
\chi_L(x) &\le p_{\rm max} \,, \qquad \text{for all}~x \in \Omega \cap \Pi_L \,, \label{eq:pMaxCond} \\
\chi_R(x) &\ge q_{\rm min} \,, \qquad \text{for all}~x \in \Omega \cap \Phi_R \,, \label{eq:qMinCond} \\
\chi_R(x) &\le q_{\rm max} \,, \qquad \text{for all}~x \in \Omega \cap \Pi_R \,. \label{eq:qMaxCond}
\end{align}
That is, any point in $\Omega \cap \Pi_L$ requires at most $p_{\rm max}$ iterations to cross $\Sigma$,
and at least $p_{\rm min}$ iterations if its preimage lies in $\Pi_R$.
The numbers $q_{\rm min}$ and $q_{\rm max}$ similarly bound the number of iterations required to cross $\Sigma$ from right to left.

For the given set $\Omega$, let
\begin{equation}
\bM_\Omega = \left\{ A_R^q A_L^p \,\big|\, p_{\rm min} \le p \le p_{\rm max},\, q_{\rm min} \le q \le q_{\rm max} \right\}.
\label{eq:bM}
\end{equation}
In Theorem \ref{th:main}, $\Omega$ is be assumed to be a {\em trapping region} for $g$,
that is, $g(\Omega) \subset {\rm int}(\Omega)$, where ${\rm int}(\cdot)$ denotes interior.
Also, to ensure a positive Lyapunov exponent, we assume $\bM_\Omega$ has a contracting-invariant, expanding cone.
This is defined as follows and illustrated in Fig.~\ref{fig:zm_cone_a}.

\begin{figure}[t!]
\begin{center}
\includegraphics[height=4.5cm]{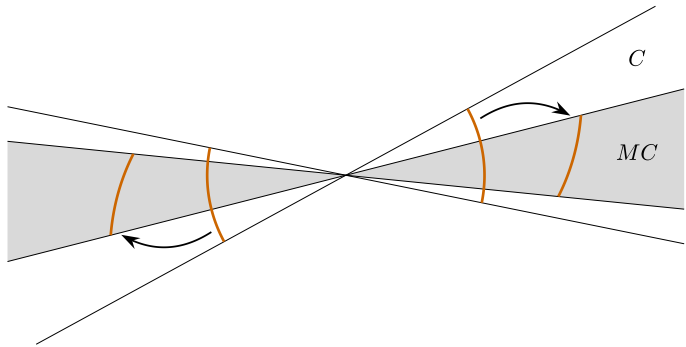}
\caption{
A sketch of a cone $C$ and its action under a matrix $M$
when $C$ is contracting-invariant and expanding for a collection $\bM$ that contains $M$.
The shaded region is the set $M C = \left\{ M v \,\middle|\, v \in C \right\}$.
The coloured curves indicate how unit vectors in $C$ map under $M$.
\label{fig:zm_cone_a}
}
\end{center}
\end{figure}

\begin{definition}
A {\em cone} is a non-empty set $C \subseteq \mathbb{R}^2$ for which $t v \in C$ for all $t \in \mathbb{R}$ and $v \in C$.
A cone $C$ is {\em contracting-invariant} for a collection of $2 \times 2$ matrices $\bM$
if $M v \in {\rm int}(C) \cup \{ \bO \}$ for all $v \in C$ and $M \in \bM$.
A cone $C$ is {\em expanding} for $\bM$ if there exists $c > 1$ such that
$\| M v \| \ge c \| v \|$ for all $v \in C$ and $M \in \bM$.
\label{df:ciec}
\end{definition}

Finally we state the main result.
We write $B_r(x)$ for the ball of radius $r > 0$ centred at $x \in \mathbb{R}^2$.

\begin{theorem}
Let $f$ be a piecewise-$C^1$ map of the form \eqref{eq:f} satisfying \eqref{eq:bcbCond}.
Let $g$ be the corresponding map \eqref{eq:bcnf} formed by using the values \eqref{eq:tLdLtRdR}.
Suppose
\begin{enumerate}
\setlength{\itemsep}{0pt}
\item
conditions \eqref{eq:transCond1}--\eqref{eq:transCond2}
and \eqref{eq:paramCond1}--\eqref{eq:paramCond3} are satisfied,
\item
$g$ has a trapping region $\Omega \subset \mathbb{R}^2$
and \eqref{eq:pMinCond}--\eqref{eq:qMaxCond} are satisfied for some
$1 \le p_{\rm min} \le p_{\rm max}$ and $1 \le q_{\rm min} \le q_{\rm max}$, and
\item
there exists a contracting-invariant, expanding cone for $\bM_\Omega$.
\end{enumerate}
Then there exists $\delta > 0$ and $s > 0$ such that for all $\mu \in (0,\delta)$
the map $f$ has a topological attractor $\Lambda_\mu \in B_{\mu s}(\bO)$
with the property that for Lebesgue almost all $y \in \Lambda_\mu$ there exists $v \in \mathbb{R}^2$ such that
\begin{equation}
\liminf_{n \to \infty} \frac{1}{n} \ln \left( \left\| \rD f^n(y) v \right\| \right) > 0.
\label{eq:liminf}
\end{equation}
\label{th:main}
\end{theorem}

Notice $\Lambda_\mu$ is chaotic in the sense of a positive Lyapunov exponent,
as implied by \eqref{eq:liminf}.

\section{A partition of the plane}
\label{sec:leftright}
\setcounter{equation}{0}

In this section we describe the geometry of the regions of the left half-plane with different values of $\chi_L$
and regions of the right half-plane with different values of $\chi_R$ for the BCNF \eqref{eq:bcnf}.
The ways in which their images intersect can be used to establish conditions
under which chaotic attractors exist \cite{GlSi21b}.
The regions corresponding to values of $\chi_L$ were described in \cite{Si20e},
so for these we simply state results without proof.
The regions corresponding to values of $\chi_R$
can be described in a similar way;
details of these calculations are provided in Appendix \ref{app:Ei}.

\begin{definition}
For each $p,q \ge 1$, let
\begin{align}
D_p &= \left\{ x \in \Pi_L \,\big|\, \chi_L(x) = p \right\}, \label{eq:Di} \\
E_q &= \left\{ x \in \Pi_R \,\big|\, \chi_R(x) = q \right\}. \label{eq:Ei}
\end{align}
\label{df:DiEi}
\end{definition}

\begin{figure}[b!]
\begin{center}
\includegraphics[height=7.5cm]{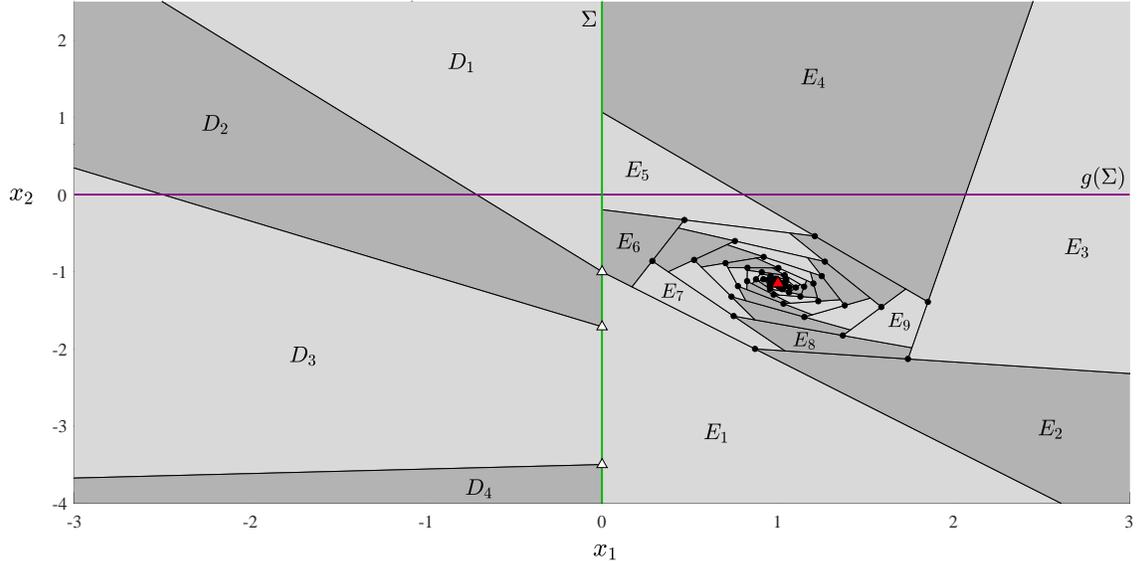}
\caption{
The regions $D_p$ \eqref{eq:Di} and $E_q$ \eqref{eq:Ei}
for a typical instance of the two-dimensional BCNF \eqref{eq:bcnf},
specifically, $(\tau_L,\delta_L,\tau_R,\delta_R) = (1.4,1,1.15,1.15)$.
In $D_p$ iterates require $p$ iterations to escape the closed left half-plane;
in $E_q$ iterates require $q$ iterations to escape the closed right half-plane
(see Definition \ref{df:chiLR}).
Here $p^* = 3$ (see Definition \ref{df:pStar}),
so $D_1,\ldots,D_4$ cover the third quadrant $\Phi_L$ (see Lemma \ref{le:DiPhiL}).
Also $q^* = 3$ and $q^{**} = 5$ (see Definition \ref{df:qStar})
so $E_3,\ldots,E_5$ cover the first quadrant $\Phi_R$ (see Lemma \ref{le:EiPhiR}).
The backwards orbit of the origin is shown with black dots.
The fixed point $x^R$ (a repelling focus) is shown with a red triangle.
\label{fig:zm_DpEq}
}
\end{center}
\end{figure}

First consider the sets $D_p$.
As shown by Proposition 6.6 of \cite{Si20e},
each $D_p$ is bounded by the lines $g_L^{-p}(\Sigma)$, $g_L^{-(p-1)}(\Sigma)$, and $\Sigma$, Fig.~\ref{fig:zm_DpEq}.
If the line $g_L^{-p}(\Sigma)$ is not vertical,
let $m_p$ denote its slope and $c_p$ denote its $x_2$-intercept, i.e.,
\begin{equation}
g_L^{-p}(\Sigma) = \left\{ x \in \mathbb{R}^2 \,\big|\, x_2 = m_p x_1 + c_p \right\}.
\label{eq:gLmiSigma}
\end{equation}

\begin{definition}
Let $p^*$ be the smallest $p \ge 1$ for which $m_p \ge 0$,
with $p^* = \infty$ if $m_p < 0$ for all $p \ge 1$.
\label{df:pStar}
\end{definition}

The next result is a simple consequence of results obtained in \cite{Si20e}.
Recall $\Phi_L$ is the third quadrant \eqref{eq:PhiL}.

\begin{lemma}
Suppose $\delta_L > 0$.
If $p^* < \infty$ then the sets $D_p \cap \Phi_L$, for $p = 1,\ldots,p^*+1$, are non-empty and cover $\Phi_L$.
If $p^* = \infty$ then $D_p \cap \Phi_L \ne \varnothing$ for all $p \ge 1$.
\label{le:DiPhiL}
\end{lemma}

For example with $(\tau_L,\delta_L) = (1.4,1)$, as in Fig.~\ref{fig:zm_DpEq}, we have $p^* = 3$.
Fig.~\ref{fig:zm_tauDeltaL} shows how the value of $p^*$,
and hence the values of $p$ for which $D_p \cap \Phi_L \ne \varnothing$,
depends on the values of $\tau_L$ and $\delta_L$.
Between curves where $m_{i-1} = 0$ and $m_i = 0$, we have
$D_p \cap \Phi_L \ne \varnothing$ if and only if $p \in \{ 1,2,\ldots i+1 \}$.
As $p \to \infty$ these curves accumulate on $\delta_L = \frac{\tau_L^2}{4}$ where
$A_L$ has repeated eigenvalues.
With $\delta_L > \frac{\tau_L^2}{4}$, $p^* = \infty$ and every $D_p \cap \Phi_L$ is non-empty.

\begin{figure}[b!]
\begin{center}
\includegraphics[height=5cm]{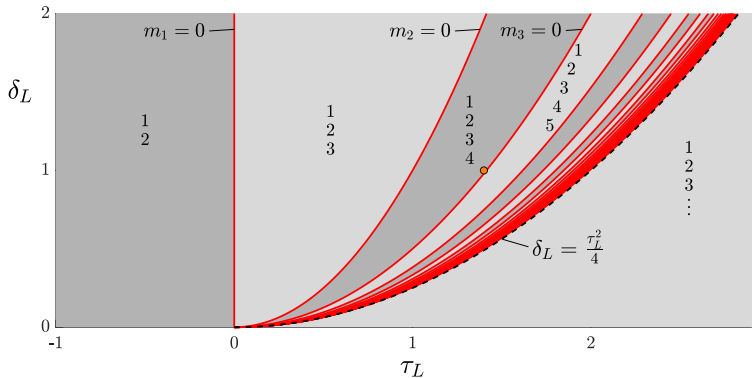}
\caption{
A division of the $(\tau_L,\delta_L)$-plane according to
values of $p$ for which $D_p \cap \Phi_L \ne \varnothing$.
The regions are labelled by these values of $p$
and are bounded by curves where $m_i = 0$ for some $i$ (see Definition \ref{df:pStar} and Lemma \ref{le:DiPhiL}).
For example, to the left of line $\tau_L = 0$
we have $D_p \cap \Phi_L \ne \varnothing$ if and only if $p \in \{ 1, 2 \}$,
so this region is labelled by the numbers $1$ and $2$.
The orange circle (located just to the left of the curve where $m_3 = 0$)
indicates the values of $\tau_L$ and $\delta_L$ used in Fig.~\ref{fig:zm_DpEq}.
\label{fig:zm_tauDeltaL}
}
\end{center}
\end{figure}

In regards to Theorem \ref{th:main},
while the values of $p_{\rm min}$ and $p_{\rm max}$ depend on the particular trapping region $\Omega$ being considered,
Lemma \ref{le:DiPhiL} tells us that the value of $p_{\rm min}$ could always be as low as $1$,
while the value of $p_{\rm max}$ cannot be more than $p^* + 1$.

Next we note that the boundaries of the $D_p$ intersect the $x_1$-axis transversally.
This is the case because each $g_L^{-p}(\Sigma)$ is a line intersecting the $x_2$-axis at $(0,c_p)$
with $c_p < 0$, see Lemma 6.2 of \cite{Si20e}.
These transversal intersections are used to argue persistence in \sect\ref{sec:proof}.

\begin{lemma}
Suppose $\delta_L > 0$ and $x = (x_1,0)$, with $x_1 < 0$, is a point on the boundary of some $D_p$.
Then, local to $x$, the boundary of $D_p$ is a line segment that intersects $g(\Sigma)$ transversally.
\label{le:DiBoundary}
\end{lemma}

We now turn our attention to the sets $E_q$.
To characterise these we assume $\tau_R$ and $\delta_R$ satisfy \eqref{eq:paramCond2} and \eqref{eq:paramCond3}.
Condition \eqref{eq:paramCond2} implies $g_R$ has the unique fixed point
\begin{equation}
x^R = \frac{1}{\delta_R - \tau_R + 1} \begin{bmatrix} 1 \\ -\delta_R \end{bmatrix},
\label{eq:xR}
\end{equation}
which lies in the fourth quadrant.
Condition \eqref{eq:paramCond3} ensures that the set of $q$ for which
$E_q \cap \Phi_R \ne \varnothing$ admits a relatively simple characterisation
and that the boundaries of the $E_q$ intersect the $x_1$-axis transversally.

By an analogous argument to that for the sets $D_p$,
the lines $g_R^{-q}(\Sigma)$, $g_R^{-(q-1)}(\Sigma)$, and $\Sigma$
form the boundaries of the $E_q$.
However, their layout is more complicated than that of the $D_p$.
Fig.~\ref{fig:zm_DpEq} shows a typical example.

If the line $g_R^{-q}(\Sigma)$ is not vertical,
let $n_q$ denote its slope and $d_q$ denote its $x_2$-intercept, i.e.,
\begin{equation}
g_R^{-q}(\Sigma) = \left\{ x \in \mathbb{R}^2 \,\big|\, x_2 = n_q x_1 + d_q \right\}.
\label{eq:gRmiSigma}
\end{equation}
If $g_R^{-q}(\Sigma)$ is vertical we write $n_q = \infty$
(in this case $g_R^{-q}(\Sigma)$ is the line $x_1 = -\frac{d_{q-1}}{\delta_R}$, see Appendix \ref{app:Ei}).

\begin{definition}
Let $q^*$ be the smallest $q \ge 1$ for which $n_q > 0$ or $n_q = \infty$.
Let $q^{**}$ be the smallest $q > q^*$ for which $d_q \le 0$.
\label{df:qStar}
\end{definition}

\begin{lemma}
Suppose \eqref{eq:paramCond2} and \eqref{eq:paramCond3} are satisfied.
Then $q^*$ and $q^{**}$ exist and
the sets $E_q \cap \Phi_R$, for $q = q^*,\ldots,q^{**}$, are non-empty and cover $\Phi_R$.
\label{le:EiPhiR}
\end{lemma}

\begin{figure}[b!]
\begin{center}
\includegraphics[height=7.5cm]{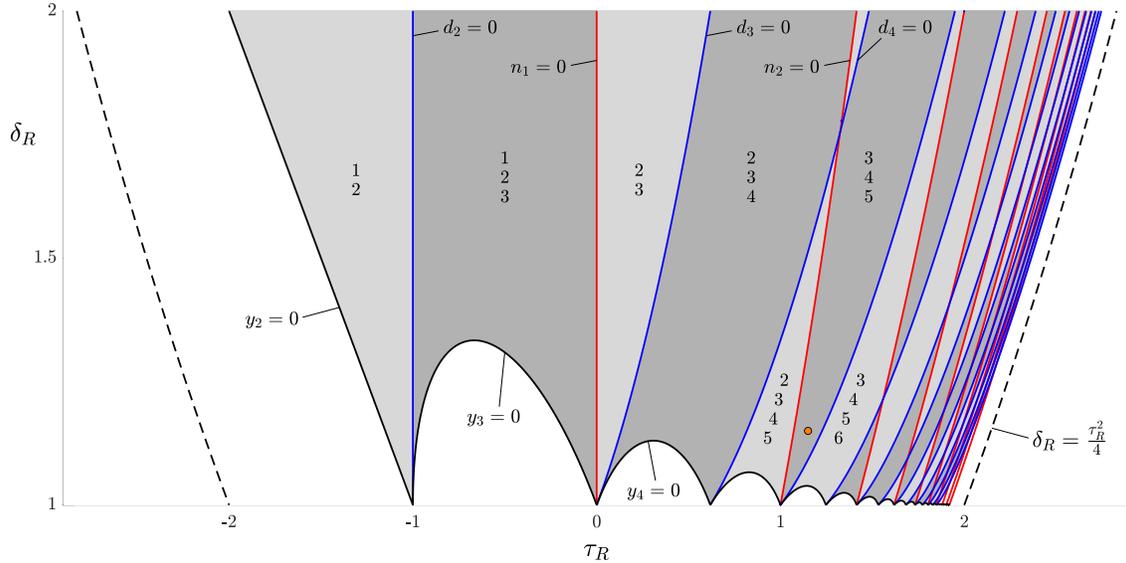}
\caption{
A division of the $(\tau_R,\delta_R)$-plane according to values of
$q$ for which $E_q \cap \Phi_R \ne \varnothing$.
These regions are labelled by these values of $q$
and bounded by curves where $n_i = 0$ (red) and $d_i = 0$ (blue) for some $i$
(see Definition \ref{df:qStar} and Lemma \ref{le:EiPhiR}).
The black curves bound the region where condition \eqref{eq:paramCond3} is satisfied
(each $y_i$ denotes $g_R^{-i}(\bO)_2$).
The orange circle indicates the values of $\tau_R$ and $\delta_R$ used in Fig.~\ref{fig:zm_DpEq}.
\label{fig:zm_tauDeltaR}
}
\end{center}
\end{figure}

Lemma \ref{le:EiPhiR} is proved in Appendix \ref{app:Ei}.
The values of $q^*$ and $q^{**}$ are determined by the
values of $\tau_R$ and $\delta_R$ as indicated in Fig.~\ref{fig:zm_tauDeltaR}.
As we move about the $(\tau_R,\delta_R)$-plane,
the value of $q^*$ changes by one when we cross a curve where $n_q = 0$,
and the value of $q^{**}$ changes by one when we cross a curve where $d_q = 0$.
These curves accumulate on $\delta_R = \frac{\tau_R^2}{4}$ past which
$A_R$ has real eigenvalues and condition \eqref{eq:paramCond2} is no longer satisfied.
In Fig.~\ref{fig:zm_tauDeltaR}, condition \eqref{eq:paramCond3} is satisfied above the piecewise-smooth black curve.

In regards to Theorem \ref{th:main},
Lemma \ref{le:EiPhiR} implies $q_{\rm min} \ge q^*$ and $q_{\rm max} \le q^{**}$.
Analogous to Lemma \ref{le:DiBoundary} we have the following result (proved in Appendix \ref{app:Ei}).

\begin{lemma}
Suppose \eqref{eq:paramCond2} and \eqref{eq:paramCond3} are satisfied
and $x = (x_1,0)$, with $x_1 > 0$, is a point on the boundary of some $E_q$.
Then, local to $x$, the boundary of $E_q$ is a line segment that intersects $g(\Sigma)$ transversally.
\label{le:EiBoundary}
\end{lemma}

\section{Proof of Theorem \ref{th:main}}
\label{sec:proof}
\setcounter{equation}{0}

The proof is divided into six steps.

\myStep{1}{Robustness of the cone.}
Let $C$ be a contracting-invariant, expanding cone for $\bM_\Omega$ of (\ref{eq:bM})
and let $c > 1$ be a corresponding expansion factor as in Definition \ref{df:ciec}.
Here we show that $C$ is an invariant expanding cone for
a collection of matrices $\hat{A}_R^q \hat{A}_L^p$
that are sufficiently close to some matrix $A_R^q A_L^p$ in $\bM_\Omega$.
Note, the expansion factor will be $\frac{c+1}{2}$ instead of $c$.

We first define a function $\mathcal{F}$ from the space of $2 \times 2$ matrices
to the space of subsets of $\mathbb{R}^2$ as follows:
given a $2 \times 2$ matrix $N$, let
\begin{equation}
\mathcal{F}(N) = \left\{ \frac{N v}{\| v \|} \,\middle|\, v \in C \setminus \{ \bO \} \right\}.
\nonumber
\end{equation}
By the definition of an contracting-invariant expanding cone,
given $M \in \bM_\Omega$ we have
$\mathcal{F}(M) \subset {\rm int}(C)$ and $\mathcal{F}(M) \cap B_c(\bO) = \varnothing$.
The function $\mathcal{F}$ is continuous, so if $N$ is sufficiently close to $M$ then
$\mathcal{F}(N) \subset C$ and $\mathcal{F}(N) \cap B_{\frac{c+1}{2}}(\bO) = \varnothing$.
That is, for all $v \in C$,
\begin{equation}
\begin{split}
N v &\in C, \\
\| N v \| &\ge \frac{c+1}{2} \| v \|.
\end{split}
\label{eq:robustCone}
\end{equation}
To be precise, there exists $\eta_1 > 0$ such that if $\| N - M \| < \eta_1$
for some $M \in \bM_\Omega$ then \eqref{eq:robustCone} is satisfied for all $v \in C$.
Furthermore, there exists $\eta_2>0$ such that $\| \hat{A}_L - A_L \| < \eta_2$ and $\| \hat{A}_R - A_R \| < \eta_2$ implies
\begin{equation}
\big\| \hat{A}_R^q \hat{A}_L^p - A_R^q A_L^p \big\| < \eta_1,
\label{eq:hatARqhatALp}
\end{equation}
for all $p_{\rm min} \le p \le p_{\rm max}$
and $q_{\rm min} \le q \le q_{\rm max}$,
and that any such $\hat{A}_L$ and $\hat{A}_R$ are invertible
(which is possible because $A_L$ and $A_R$ are invertible by \eqref{eq:paramCond1} and \eqref{eq:paramCond2}).

\myStep{2}{Bounds related to $g(\Omega)$.}
By the definition of $p_{\rm min}$, $p_{\rm max}$, $q_{\rm min}$, and $q_{\rm max}$,
see \eqref{eq:pMinCond}--\eqref{eq:qMaxCond},
we have $\Omega \cap \Phi_L \subset \bigcup_{p = p_{\rm min}}^{p_{\rm max}} D_p$
and $\Omega \cap \Phi_R \subset \bigcup_{q = q_{\rm min}}^{q_{\rm max}} E_q$.
In this step we use the results of \sect\ref{sec:leftright} and the
fact that $\Omega$ maps to its interior under $g$
to control the behaviour of points inside and near $g(\Omega)$.

Since $g(\Omega) \subset {\rm int}(\Omega)$ is compact
there exists $\ee_1 > 0$ such that 
\begin{equation}
B_{\ee_1}(g(\Omega)) \subset \Omega,
\label{eq:ee1Omega}
\end{equation}
as illustrated in Fig.~\ref{fig:zm_gHatOmega}.
Now consider the set 
$U = B_{\ee_1} \left( g(\Omega) \cap \Phi_L \right) \cap \Pi_L$.
This set is contained in $\Omega \cap \Phi_L$ except has some points just above the negative $x_1$-axis.
By Lemma \ref{le:DiBoundary}, we can assume $\ee_1$ is small enough
that $U$ does not intersect any `other' regions $D_p$, that is
$U \subset \bigcup_{p = p_{\rm min}}^{p_{\rm max}} D_p$.
Further, by shrinking this set by a small amount, the result will be bounded away from the other regions $D_p$.
That is, there exists $\ee_2 > 0$ such that
for all $x \in B_{\frac{\ee_1}{2}} \left( g(\Omega) \cap \Phi_L \right) \cap \Pi_L$:
\begin{enumerate}
\setlength{\itemsep}{0pt}
\item
$g_L^p(x)_1 < -\ee_2$ for all $p = 1,\ldots,p_{\rm min}-1$, and
\item
if $g_L^p(x)_1 \le 0$ for all $p = p_{\rm min},\ldots,p_{\rm max}-1$, then $g_L^{p_{\rm max}}(x)_1 > \ee_2$.
\end{enumerate}
Also, $\Omega \cap \Pi_L \subset \bigcup_{p = 1}^{p_{\rm max}} D_p$ by \eqref{eq:pMaxCond}.
Thus by \eqref{eq:ee1Omega} we can assume $\ee_2$ is small enough that for all
$x \in B_{\frac{\ee_1}{2}}(g(\Omega))$:
\begin{enumerate}
\setlength{\itemsep}{0pt}
\addtocounter{enumi}{2}
\item
if $g_L^p(x)_1 \le 0$ for all $p = 1,\ldots,p_{\rm max}-1$, then $g_L^{p_{\rm max}}(x)_1 > \ee_2$.
\end{enumerate}
In view of Lemma \ref{le:EiBoundary} we can assume $\ee_1$ and $\ee_2$ are small enough
that analogous bounds also hold for $g_R^q(x)_1$.

\begin{figure}[b!]
\begin{center}
\includegraphics[height=6cm]{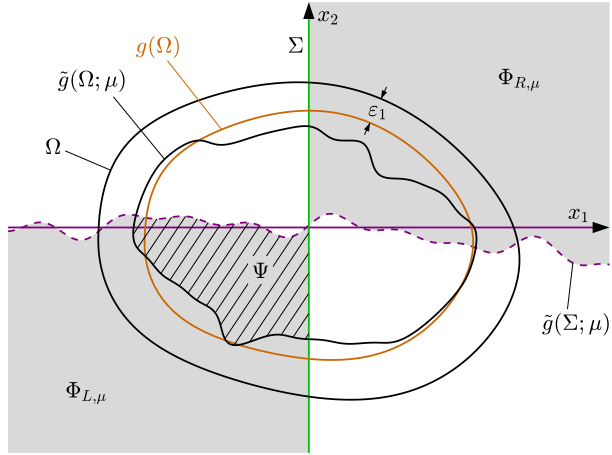}
\caption{
A sketch of $\Omega$ and its image $g(\Omega)$ illustrating the bound \eqref{eq:ee1Omega}.
For the perturbed map $\tilde{g}$ we also sketch the regions $\Phi_{L,\mu}$ and $\Phi_{R,\mu}$ (shaded) introduced in Step 4.
The part of $\tilde{g}(\Omega;\mu)$ that lies in $\Phi_{L,\mu}$ is the set $\Psi$ (striped) introduced in Step 5.
\label{fig:zm_gHatOmega}
}
\end{center}
\end{figure}

\myStep{3}{Change of coordinates.}
Here we apply the coordinate change \eqref{eq:ytox} for converting $f$ to $g$ plus higher order terms.
For small $\mu > 0$ this coordinate change represents a spatial blow-up of phase space near the origin.
Since the higher order terms are small near the origin
we are able to control the higher order terms by assuming $\mu$ is sufficiently small.
Here we also let $r > 0$ be such that $\Omega \subset B_r(\bO)$.

We first decompose \eqref{eq:ytox} into the spatial blow-up
\begin{equation}
x = \frac{\tilde{y}}{\gamma \mu},
\label{eq:blowUp}
\end{equation}
and the bounded coordinate change
\begin{equation}
\tilde{y} = \phi_\mu(y) = \begin{bmatrix}
y_1 \\ -a_{22} y_1 + a_{12} y_2 + (a_{22} b_1 - a_{12} b_2) \mu
\end{bmatrix}.
\label{eq:coordChange}
\end{equation}
Notice \eqref{eq:coordChange} is invertible because $a_{12} \ne 0$, \eqref{eq:transCond1}.
The coordinate change \eqref{eq:coordChange} transforms $f$ to a map
\begin{equation}
\tilde{f}(\tilde{y};\mu) = \begin{cases}
\tilde{f}_L(\tilde{y};\mu), & \tilde{y}_1 \le 0, \\
\tilde{f}_R(\tilde{y};\mu), & \tilde{y}_1 \ge 0,
\end{cases}
\label{eq:fTilde}
\end{equation}
where
\begin{equation}
\begin{split}
\tilde{f}_L(\tilde{y};\mu) &= A_L \tilde{y} + \begin{bmatrix} \gamma \mu \\ 0 \end{bmatrix} + E_L(\tilde{y};\mu), \\
\tilde{f}_R(\tilde{y};\mu) &= A_R \tilde{y} + \begin{bmatrix} \gamma \mu \\ 0 \end{bmatrix} + E_R(\tilde{y};\mu).
\end{split}
\label{eq:fTildeLR}
\end{equation}
The higher order terms $E_L$ and $E_R$ are $C^1$ and $\co \left( \| \tilde{y} \| + |\mu| \right)$.
Thus given any $\ee > 0$ there exists $\delta = \delta(\ee) > 0$ such that
\begin{equation}
\begin{split}
\frac{\left\| E_L(\tilde{y};\mu) \right\|}{\| \tilde{y} \| + \mu} &< \frac{\gamma \ee}{\gamma r + 1},
\qquad \text{for all}~ (\tilde{y};\mu) \in B_{\delta \gamma r}(\bO) \times (0,\delta), \\
\frac{\left\| E_R(\tilde{y};\mu) \right\|}{\| \tilde{y} \| + \mu} &< \frac{\gamma \ee}{\gamma r + 1},
\qquad \text{for all}~ (\tilde{y};\mu) \in B_{\delta \gamma r}(\bO) \times (0,\delta),
\end{split}
\label{eq:ELERbounds}
\end{equation}
and
\begin{equation}
\begin{split}
\big\| \rD \tilde{f}_L(\tilde{y};\mu) - A_L \big\| &< \eta_2 \,,
\qquad \text{for all}~ (\tilde{y};\mu) \in B_{\delta \gamma r}(\bO) \times (0,\delta), \\
\big\| \rD \tilde{f}_R(\tilde{y};\mu) - A_R \big\| &< \eta_2 \,,
\qquad \text{for all}~ (\tilde{y};\mu) \in B_{\delta \gamma r}(\bO) \times (0,\delta).
\end{split}
\label{eq:DtildefLRbounds}
\end{equation}
Then, since $\gamma > 0$ \eqref{eq:transCond2}, the spatial blow-up \eqref{eq:blowUp}
converts \eqref{eq:fTilde} with $\mu > 0$ to a map of the form
\begin{equation}
\tilde{g}(x;\mu) = \begin{cases}
\tilde{g}_L(x;\mu), & x_1 \le 0, \\
\tilde{g}_R(x;\mu), & x_1 \ge 0,
\end{cases}
\label{eq:gHat}
\end{equation}
where $\tilde{g}_L(x;\mu) = \frac{1}{\gamma \mu} \tilde{f}_L(\gamma \mu x;\mu)$
and $\tilde{g}_R(x;\mu) = \frac{1}{\gamma \mu} \tilde{f}_R(\gamma \mu x;\mu)$.
By \eqref{eq:fTildeLR}--\eqref{eq:DtildefLRbounds} we have
\begin{equation}
\begin{split}
\left\| \tilde{g}_L(x;\mu) - g_L(x) \right\| &< \ee,
\qquad \text{for all}~ (x;\mu) \in B_r(\bO) \times (0,\delta), \\
\left\| \tilde{g}_R(x;\mu) - g_R(x) \right\| &< \ee,
\qquad \text{for all}~ (x;\mu) \in B_r(\bO) \times (0,\delta),
\end{split}
\label{eq:hatgLRbounds}
\end{equation}
and
\begin{equation}
\begin{split}
\left\| \rD \tilde{g}_L(x;\mu) - A_L \right\| &< \eta_2 \,,
\qquad \text{for all}~ (x;\mu) \in B_r(\bO) \times (0,\delta), \\
\left\| \rD \tilde{g}_R(x;\mu) - A_R \right\| &< \eta_2 \,,
\qquad \text{for all}~ (x;\mu) \in B_r(\bO) \times (0,\delta).
\end{split}
\label{eq:DhatgLRbounds}
\end{equation}
Further assume $\ee \le \ee_1$ so that, by \eqref{eq:ee1Omega},
$\Omega$ is a trapping region for $\tilde{g}(x;\mu)$ with any $\mu \in (0,\delta)$.
That $\Omega$ is a trapping region implies $\tilde{g}$ has a topological attractor $\Gamma_\mu \subset \Omega$.
Then $\Lambda_\mu = \phi_\mu^{-1} \left( \gamma \mu \Gamma_\mu \right)$ is the corresponding attractor of $f$.
Since $\gamma \mu \Gamma_\mu \subset B_{\gamma \mu r}(\bO)$
and $\phi_\mu^{-1}(\tilde{y})$ is a linear function of the pair $(\tilde{y};\mu)$,
there exists $s > 0$ such that $\Lambda_\mu \subset B_{\mu s}(\bO)$ for all $\mu \in (0,\delta)$.

\myStep{4}{Extend bounds on $g$ to the perturbed map $\tilde{g}$.}
In Step 2 we obtained bounds on the number of iterations
required for orbits of the BCNF $g$ to cross $\Sigma$.
Here we show the same bounds hold for the perturbed map $\tilde{g}$.

We first extend the definitions of $\chi_L$ and $\chi_R$ to $\tilde{g}$.
Given $x \in \mathbb{R}^2$,
let $\chi_{L,\mu}(x)$ be the smallest $p \ge 1$ for which $\tilde{g}^p(x;\mu) \notin \Pi_L$ and
let $\chi_{R,\mu}(x)$ be the smallest $q \ge 1$ for which $\tilde{g}^q(x;\mu) \notin \Pi_R$,
if such $p$ and $q$ exist.
In view of Lemma \ref{le:crossingSigma},
we can similarly generalise $\Phi_L$ and $\Phi_R$ by defining
\begin{equation}
\begin{split}
\Phi_{L,\mu} &= \left\{ x \in \mathbb{R}^2 \,\big|\, x_1 < 0,\, \tilde{g}^{-1}(x;\mu)_1 \ge 0 \right\}, \\
\Phi_{R,\mu} &= \left\{ x \in \mathbb{R}^2 \,\big|\, x_1 > 0,\, \tilde{g}^{-1}(x;\mu)_1 \le 0 \right\}.
\end{split}
\label{eq:PhiLRmu}
\end{equation}
These sets are sketched in Fig.~\ref{fig:zm_gHatOmega}
and for sufficiently small $\mu < 0$ they are within
$\frac{\ee_1}{2}$ of $\Phi_L$ and $\Phi_R$ in the bounded set $\Omega$.
Next, take $0 < \ee < \frac{\ee_1}{2}$ small enough so that, by \eqref{eq:hatgLRbounds},
$\tilde{g}(\Omega;\mu) \cap \Phi_{L,\mu} \subset B_{\frac{\ee_1}{2}} \left( g(\Omega) \cap \Phi_L \right)$ and
$\tilde{g}(\Omega;\mu) \cap \Phi_{R,\mu} \subset B_{\frac{\ee_1}{2}} \left( g(\Omega) \cap \Phi_R \right)$
for all $\mu \in (0,\delta)$.
For small enough $\ee > 0$, \eqref{eq:hatgLRbounds} also implies
\begin{equation}
\begin{split}
\left\| \tilde{g}_L^p(x;\mu) - g_L^p(x) \right\| &< \ee_2,
\qquad \text{for all}~ (x;\mu) \in \Omega \times (0,\delta) ~\text{and all}~ p = 1,2,\ldots,p_{\rm max} \,, \\
\left\| \tilde{g}_R^q(x;\mu) - g_R^q(x) \right\| &< \ee_2,
\qquad \text{for all}~ (x;\mu) \in \Omega \times (0,\delta) ~\text{and all}~ q = 1,2,\ldots,q_{\rm max} \,.
\end{split}
\label{eq:hatgLRbounds2}
\end{equation}
Then by (i) and (ii) of Step 2, and analogous bounds on $g_R^q(x)_1$,
\begin{equation}
\begin{split}
p_{\rm min} &\le \chi_{L,\mu}(x) \le p_{\rm max}, \qquad
\text{for all}~ (x;\mu) \in \tilde{g}(\Omega;\mu) \cap \Phi_{L,\mu} \times (0,\delta), \\
q_{\rm min} &\le \chi_{R,\mu}(x) \le q_{\rm max}, \qquad
\text{for all}~ (x;\mu) \in \tilde{g}(\Omega;\mu) \cap \Phi_{L,\mu} \times (0,\delta).
\end{split}
\label{eq:chiLRmubounds}
\end{equation}

\myStep{5}{Construct an induced map $F$.}
Let
\begin{equation}
\Psi = \tilde{g}(\Omega;\mu) \cap \Phi_{L,\mu} \,,
\nonumber
\end{equation}
as indicated in Fig.~\ref{fig:zm_gHatOmega}.
In this step we introduce an {\em induced map} $F : \Psi \to \Psi$
that provides the first return to $\Psi$ under iterations of $\tilde{g}$.	

Fix $\mu \in (0,\delta)$.
Points in $\Psi$
map to $\tilde{g}(\Omega;\mu) \cap \Phi_{R,\mu}$ under $\chi_{L,\mu}$ iterations of $\tilde{g}_L$.
Similarly, points in $\tilde{g}(\Omega;\mu) \cap \Phi_{R,\mu}$ map to
$\Psi$ under $\chi_{R,\mu}$ iterations of $\tilde{g} = \tilde{g}_R$.
Consequently we can define $F : \Psi \to \Psi$ by
\begin{equation}
F(x) = \tilde{g}_R^q \big( \tilde{g}_L^p(x;\mu); \mu \big),
\label{eq:F}
\end{equation}
where $p = \chi_{L,\mu}(x)$
and $q = \chi_{R,\mu} \left( \tilde{g}_L^p(x;\mu) \right)$.
Let $\Sigma_\infty = \left\{ x \in \mathbb{R}^2 \,\middle|\, \tilde{g}^i(x) \in \Sigma \text{~for some~} i \ge 0 \right\}$
be the set of all points whose forward orbits under $\tilde{g}$ intersect $\Sigma$.
Then $\rD F(x;\mu)$ is well-defined at any $x \in \Psi \setminus \Sigma_\infty$.
By \eqref{eq:hatARqhatALp} and \eqref{eq:DhatgLRbounds},
\begin{equation}
\left\| \rD F(x;\mu) - A_R^q A_L^p \right\| < \eta_1 \,,
\label{eq:DFbound}
\end{equation}
where $p$ and $q$ are as in \eqref{eq:F}.
Note \eqref{eq:DFbound} also relies on the bounds
$p_{\rm min} \le p \le p_{\rm max}$
and $q_{\rm min} \le q \le q_{\rm max}$ provided by \eqref{eq:chiLRmubounds}.

\myStep{6}{Bound the Lyapunov exponent.}
Finally we verify \eqref{eq:liminf}.
Choose any $x \in \Omega \setminus \Sigma_\infty$.
We have $\tilde{g}(x;\mu) \in B_{\frac{\ee_1}{2}}(g(\Omega))$ by \eqref{eq:hatgLRbounds} because $\ee \le \frac{\ee_1}{2}$.
Then by \eqref{eq:hatgLRbounds2} and (i)--(iii) of Step 2,
there exists $k \le p_{\rm max} + q_{\rm max} + 1$ such that $\tilde{g}^k(x;\mu) \in \Psi$.
Let $x^{(0)} = \tilde{g}^k(x;\mu)$.
Also let $u^{(0)} \in C \setminus \{ \bO \}$
and $u = \left( \rD \tilde{g}^k(x;\mu) \right)^{-1} u^{(0)}$
(the inverse exists by the last remark in Step 1).

For each $j \ge 1$, let
$x^{(j)} = F \left( x^{(j-1)} \right)$
and let $p_{j-1}$ and $q_{j-1}$ be the corresponding $p$ and $q$ values in \eqref{eq:F}.
Then for all $j \ge 1$ we have $x^{(j)} = \tilde{g}^{n_j}(x;\mu)$
where $n_j = k + (p_0 + q_0) + (p_1 + q_1) + \cdots + (p_{j-1} + q_{j-1})$.
For all $j \ge 1$, let $u^{(j)} = \rD F \left( x^{(j-1)} \right) u^{(j-1)}$.
By \eqref{eq:robustCone} and \eqref{eq:DFbound} and an inductive argument on $j$,
we have $u^{(j)} \in C$ and $\| u^{(j)} \| > \frac{c+1}{2} \| u^{(j-1)} \|$
for all $j \ge 1$.
Hence $\| u^{(j)} \| > \left( \frac{c+1}{2} \right)^j \| u^{(0)} \|$.
Since $u^{(j)} = \rD \tilde{g}^{n_j}(x;\mu) u$, we have
\begin{align}
\liminf_{n \to \infty} \frac{1}{n} \ln \left( \left\| \rD \tilde{g}^n(x;\mu) u \right\| \right)
&= \liminf_{j \to \infty} \frac{1}{n_j} \ln \left( \left\| u^{(j)} \right\| \right) \nonumber \\
&\ge \liminf_{j \to \infty} \frac{1}{k + j (p_{\rm max} + q_{\rm max})} \ln \left( \left( \tfrac{c+1}{2} \right)^j \| u^{(0)} \| \right) \nonumber \\
&= \frac{\ln \left( \frac{c+1}{2} \right)}{p_{\rm max} + q_{\rm max}} > 0. \nonumber
\end{align}
Since the coordinate transformation $\phi_\mu$ is invertible,
the same bound applies to $f$ with $y = \phi^{-1}_\mu(\gamma \mu x)$
and $v = \left( \rD \phi_\mu(y) \right)^{-1} u$, i.e.~\eqref{eq:liminf}.
Since this applies to any $x \in \Omega \setminus \Sigma_\infty$,
where $\Sigma_\infty$ has zero Lebesgue measure,
\eqref{eq:liminf} holds for Lebesgue almost all $y \in \Lambda_\mu$.
\hfill $\Box$

\section{An application to power converters}
\label{sec:application}
\setcounter{equation}{0}

Power converters take a raw input voltage and use control strategies
to produce an output that is close to a desired voltage \cite{BaVe01,Ts03}.
These have applications in many areas including personal electronic equipment
where the voltage from a domestic electricity supplier is different from the voltage required by the device.
A prototypical DC/DC power converter model described in \cite{ZhMo03,ZhMo08b}
is written in terms of time-dependent variables $X(t)$ and $Y(t)$ that represent linear combinations
of an internal current and voltage as
\begin{equation}
\begin{split}
\frac{d X}{d t} &= \lambda_1 \big( X - H(\xi(\lfloor t \rfloor) - \eta(t)) \big), \\
\frac{d Y}{d t} &= \lambda_2 \big( Y - H(\xi(\lfloor t \rfloor) - \eta(t)) \big),
\end{split}
\label{eq:odes}
\end{equation}
where $H$ is the Heaviside function and
\begin{align}
\xi(t) &= X(t) - \theta Y(t) + \frac{q}{2 \omega}, \label{eq:xi} \\
\eta(t) &= \frac{q}{\alpha \omega} \big( t - \lfloor t \rfloor \big). \label{eq:eta} 
\end{align}
The function $\xi(t)$ is the control signal employed by the converter.
In this model time has been scaled so that the switching period is $1$, see \sect 5.2 of \cite{ZhMo03} for more details.
The floor function $\lfloor t \rfloor$ denotes the largest integer less than or equal to $t$.

Here we fix
\begin{equation}
\begin{split}
\lambda_1 &= -0.977, \\
\lambda_2 &= -0.232, \\
q &= 35.606, \\
\theta &= 4.2, \\
\alpha &= 70,
\end{split}
\label{eq:ZhMo08b_params}
\end{equation}
and vary the value of $\omega$
which represents input voltage and is a controllable parameter.
The values \eqref{eq:ZhMo08b_params} are as given in \cite{ZhMo08b}
except we have used a slightly larger value of $\alpha$
so that the border-collision bifurcation produces a chaotic attractor instead of an invariant torus.

\begin{figure}[b!]
\begin{center}
\includegraphics[height=6cm]{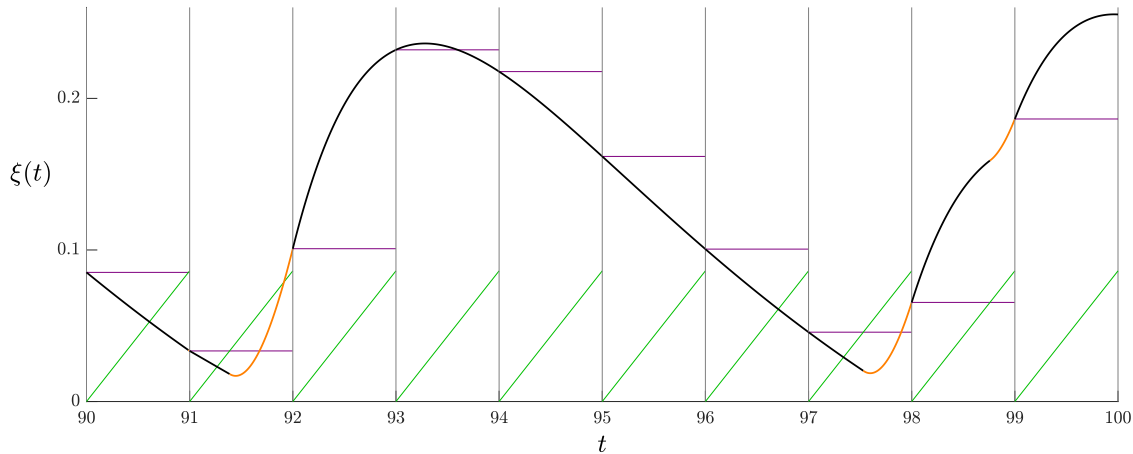}
\caption{
A time series of \eqref{eq:odes} with \eqref{eq:ZhMo08b_params} and $\omega = 5.9$
showing the control signal $\xi(t)$ on the vertical axis.
Also $\xi(\lfloor t \rfloor)$ is shown in purple and $\eta(t)$ is shown in green.
Where $\xi(\lfloor t \rfloor) > \eta(t)$ we have $H = 1$ in \eqref{eq:odes} and the time series is coloured black;
where $\xi(\lfloor t \rfloor) < \eta(t)$ we have $H = 0$ in \eqref{eq:odes} and the time series is coloured orange.
\label{fig:zm_timeSeries}
}
\end{center}
\end{figure}

Fig.~\ref{fig:zm_timeSeries} shows a typical time series of \eqref{eq:odes}.
On the vertical axis we have plotted $\xi(t)$ which is an affine function of the variables.
The system switches from $H=1$ (black) to $H=0$ (orange)
when the purple line $\xi(\lfloor t \rfloor)$ meets the green line $\eta(t)$, and switches back to $H=1$ at integer times.

We now provide a stroboscopic map that captures the dynamics of \eqref{eq:odes}.
This map is given in \cite{ZhMo08b} and is straight-forward to derive.
Let $w^{(n)} = (X(n),Y(n))$, for $n \in \mathbb{Z}$, denote the solution to \eqref{eq:odes} at integer times.
The stroboscopic map, $w^{(n+1)} = p \left( w^{(n)} \right)$, is
\begin{equation}
p(w) =
\begin{bmatrix} \re^{\lambda_1}(w_1-1) + \re^{\lambda_1(1-z)} \\ \re^{\lambda_2}(w_2-1) + \re^{\lambda_2(1-z)} \end{bmatrix},
\label{eq:stroboscopicMap}
\end{equation}
where
\begin{equation}
z = \begin{cases}
0, & \varphi \le 0, \\
\frac{\alpha \omega \varphi}{q}, & 0 \le \varphi \le \frac{q}{\alpha \omega}, \\
1, & \varphi \ge \frac{q}{\alpha \omega},
\end{cases}
\label{eq:z}
\end{equation}
and
\begin{equation}
\varphi = w_1 - \theta w_2 + \frac{q}{2 \omega}.
\label{eq:varphi}
\end{equation}
Fig.~\ref{fig:zm_bifDiag} shows a bifurcation diagram of \eqref{eq:stroboscopicMap}.
As the value of $\omega$ is increased,
a stable fixed point undergoes a border-collision bifurcation at
\begin{equation}
\omega_{\rm BCB} = \frac{q}{1 - \theta} \left( \frac{1}{\alpha} - \frac{1}{2} \right) \approx 5.4045.
\label{eq:OmegaBCB}
\end{equation}
Numerical simulations suggest that a chaotic attractor is created in the border-collision bifurcation.
The numerically computed Lyapunov exponent remains positive until $\omega \approx 5.555$.
Fig.~\ref{fig:zm_qqStrobo}(a) shows a phase portrait of the chaotic attractor.

\begin{figure}[b!]
\begin{center}
\includegraphics[height=7.5cm]{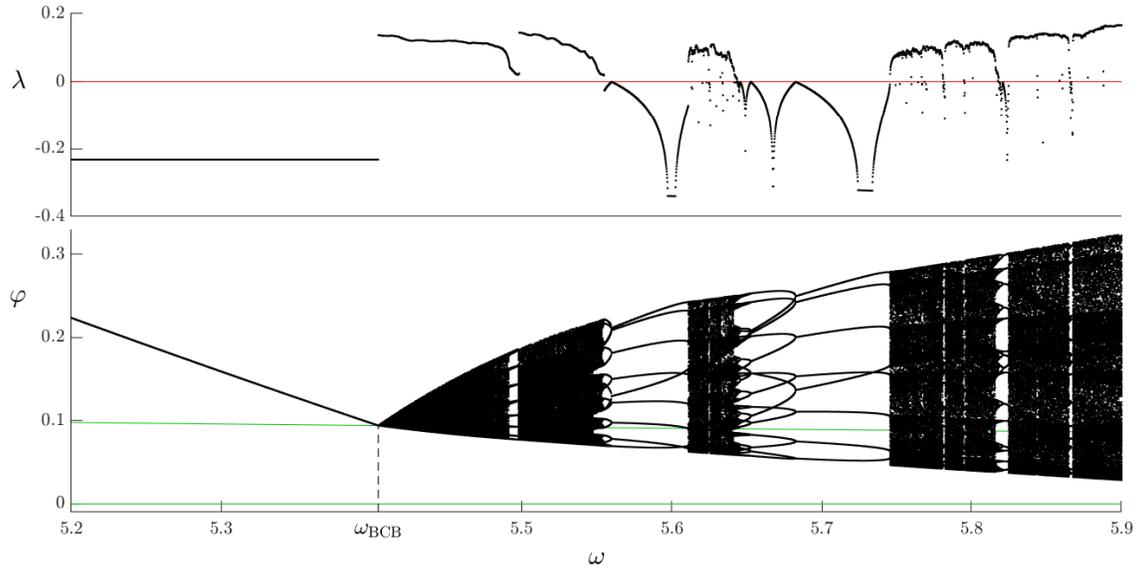}
\caption{
A bifurcation diagram (lower plot) of \eqref{eq:stroboscopicMap} with \eqref{eq:ZhMo08b_params}
and a numerically computed Lyapunov exponent $\lambda$ (upper plot).
A border-collision bifurcation occurs at $\omega = \omega_{\rm BCB} \approx 5.4045$.
For $5000$ different values of $\omega$ the Lyapunov exponent was estimated from $10^6$ iterates of one orbit.
\label{fig:zm_bifDiag}
}
\end{center}
\end{figure}

\begin{figure}[b!]
\begin{center}
\setlength{\unitlength}{1cm}
\begin{picture}(12,5)
\put(0,0){\includegraphics[height=5cm]{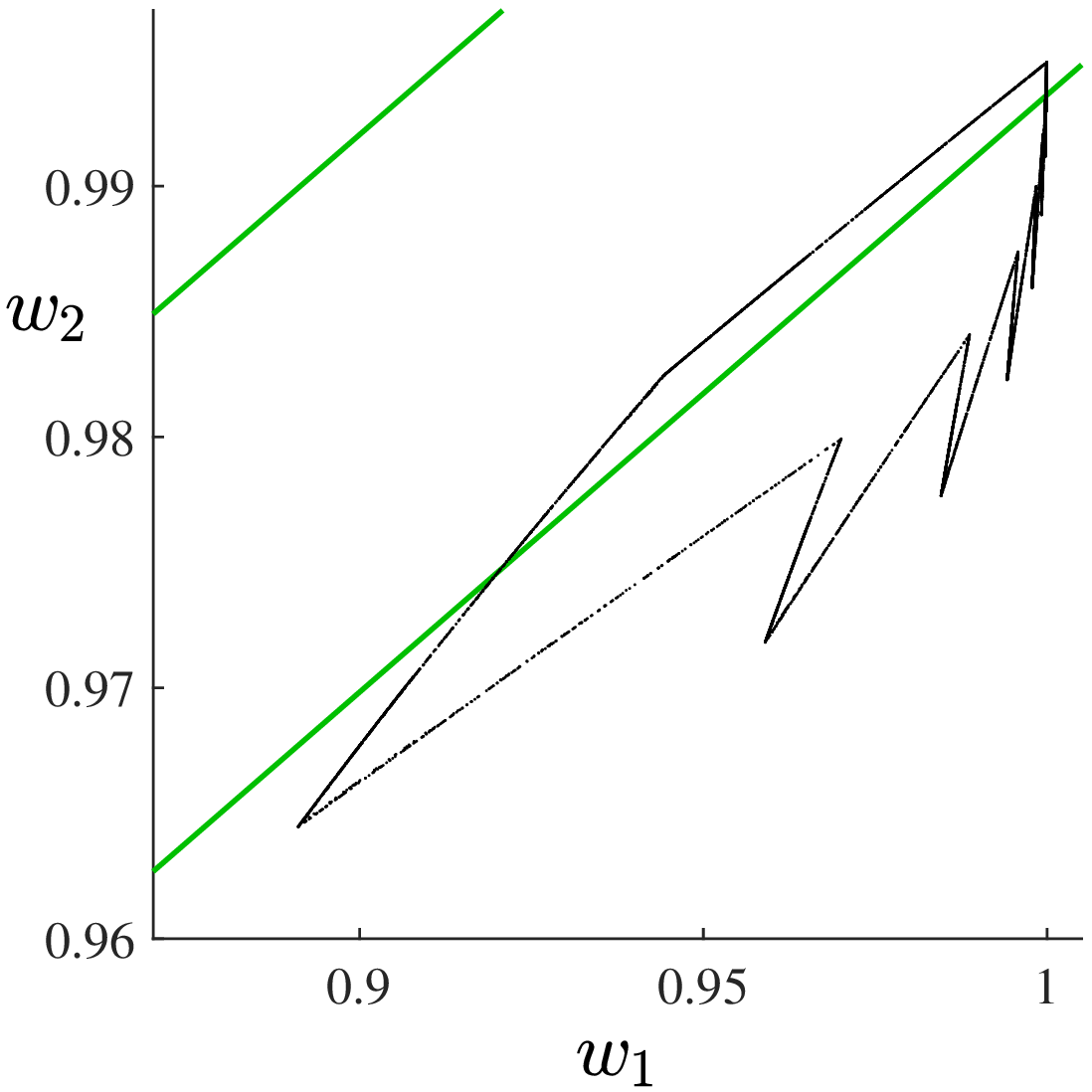}}
\put(7,0){\includegraphics[height=5cm]{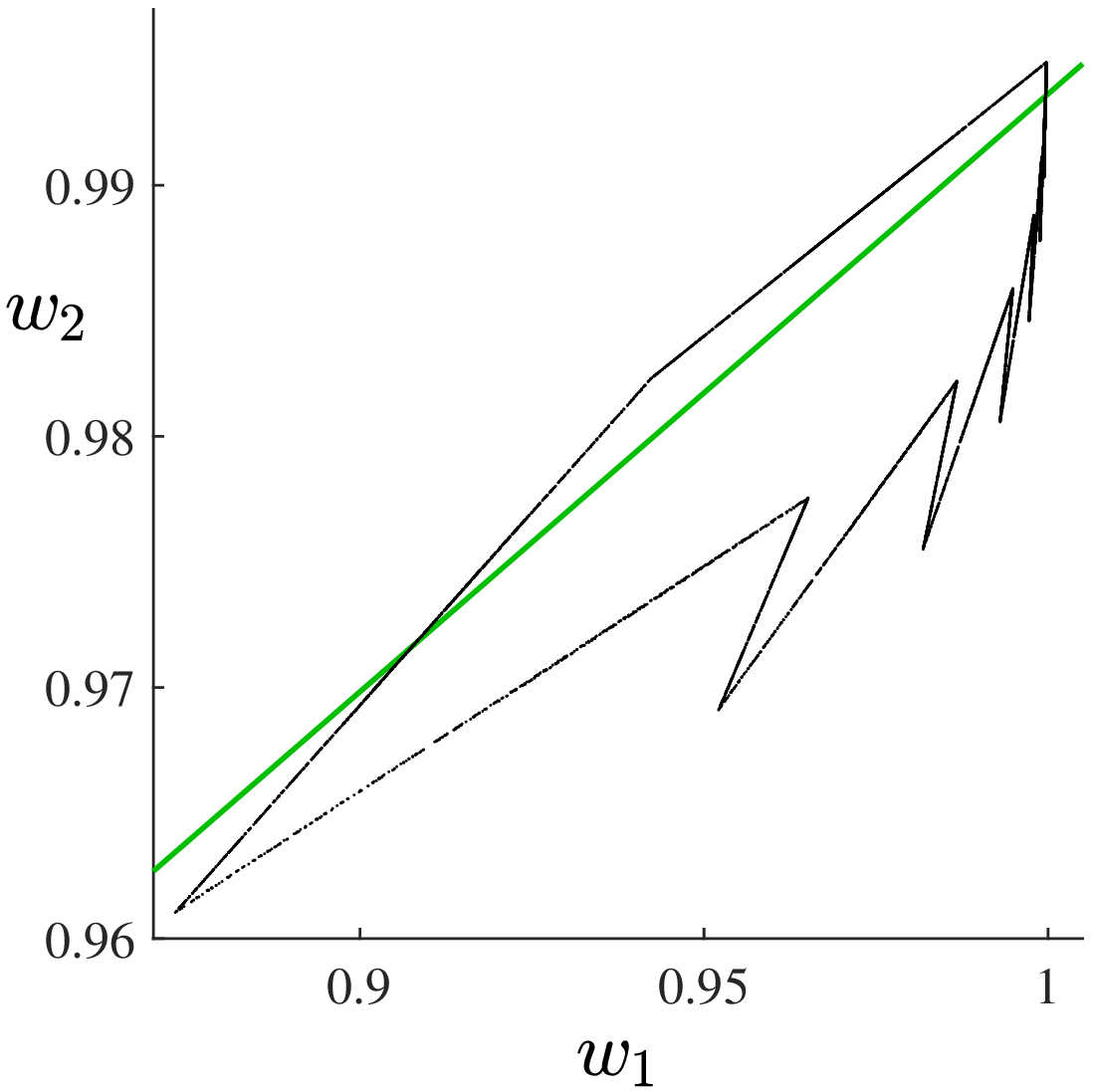}}
\put(0,4.8){\bf a)}
\put(7,4.8){\bf b)}
\end{picture}
\caption{
Panel (a) shows a phase portrait of \eqref{eq:stroboscopicMap} with \eqref{eq:ZhMo08b_params} and $\omega = 5.45$.
Specifically we show $10^4$ points of a forward orbit with the first $100$ points removed.
Panel (b) similarly shows part of the forward orbit of the corresponding two-dimensional BCNF,
except converted to the coordinates of the stroboscopic map.
\label{fig:zm_qqStrobo}
}
\end{center}
\end{figure}

The map \eqref{eq:stroboscopicMap} has two switching manifolds.
The border-collision bifurcation occurs on the switching manifold $\varphi = \frac{q}{\alpha \omega}$,
so for the purposes of analyzing the local dynamics associated with this bifurcation we can ignore the
$\varphi \le 0$ component of \eqref{eq:z}.
Upon performing the affine change of variables
\begin{equation}
y = \begin{bmatrix} \frac{q}{\alpha \omega} - \left( w_1 - \theta w_2 + \frac{q}{2 \omega} \right) \\ 1 - w_2 \end{bmatrix},
\label{eq:changeToGeneralMap}
\end{equation}
the map (without the $\varphi \le 0$ component of \eqref{eq:z}) takes the form \eqref{eq:f}.
Also let
\begin{equation}
\mu = \omega - \omega_{\rm BCB} \,,
\label{eq:changeToGeneralMap2}
\end{equation}
which with \eqref{eq:bcbCond} is satisfied,
i.e.~the border-collision bifurcation occurs at $y = \bO$ when $\mu = 0$.
By differentiating \eqref{eq:stroboscopicMap} and evaluating it at the border-collision bifurcation we obtain
\begin{equation}
\begin{split}
\tau_L &= \re^{\lambda_1} + \re^{\lambda_2}, \\
\delta_L &= \re^{\lambda_1 + \lambda_2}, \\
\tau_R &= \re^{\lambda_1} + \re^{\lambda_2} - \frac{\frac{\alpha}{2} - 1}{\theta - 1} \big( \lambda_1 - \theta \lambda_2 \big), \\
\delta_R &= \re^{\lambda_1 + \lambda_2} - \frac{\frac{\alpha}{2} - 1}{\theta - 1} \left( \lambda_1 \re^{\lambda_2} - \theta \lambda_2 \re^{\lambda_1} \right).
\end{split}
\label{eq:tLdLtRdRFormulas}
\end{equation}
By substituting \eqref{eq:ZhMo08b_params} and \eqref{eq:OmegaBCB} into \eqref{eq:tLdLtRdRFormulas} and rounding to four decimal places we obtain
\begin{equation}
\begin{split}
\tau_L &= 1.1694, \\
\delta_L &= 0.2985, \\
\tau_R &= 1.1970, \\
\delta_R &= 4.6325.
\end{split}
\label{eq:tLdLtRdRValues}
\end{equation}

\begin{figure}[b!]
\begin{center}
\includegraphics[height=5cm]{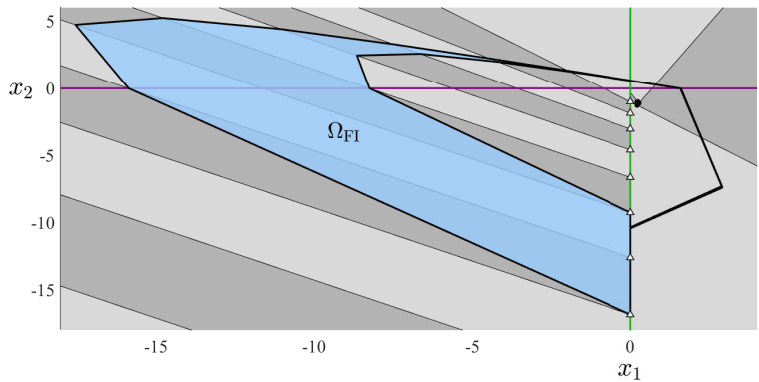}
\caption{
A forward invariant region $\Omega_{\rm FI}$ for the BCNF \eqref{eq:bcnf} with parameter values \eqref{eq:tLdLtRdRValues}
corresponding to the border-collision bifurcation of the power converter model.
The regions $D_p$ and $E_q$ are shaded as in Fig.~\ref{fig:zm_DpEq}.
\label{fig:zm_qqBCNF}
}
\end{center}
\end{figure}

We now show that this border-collision bifurcation satisfies the conditions of Theorem \ref{th:main}.
Certainly \eqref{eq:transCond1} is satisfied;
\eqref{eq:transCond2} is also satisfied due to signs choices made when constructing \eqref{eq:changeToGeneralMap}.
Conditions \eqref{eq:paramCond1}--\eqref{eq:paramCond3} are satisfied
with $p^* = \infty$ in Definition \ref{df:pStar} (the eigenvalues of $A_L$ are complex)
and $q^* = 2$ and $q^{**} = 3$ in Definition \ref{df:qStar}.

With \eqref{eq:tLdLtRdRValues} the BCNF $g$
has the forward invariant region $\Omega_{\rm FI}$ shown in Fig.~\ref{fig:zm_qqBCNF}.
This region was constructed using the algorithm of \cite{GlSi21b}
(specifically $\Omega_{\rm FI}$ is the union of images of a `recurrent' region $\Omega_{\rm rec}$).
For $\Omega_{\rm FI}$ we have
\begin{align}
p_{\rm min} &= 6, &
p_{\rm max} &= 8, &
q_{\rm min} &= 2, &
q_{\rm max} &= 3,
\label{eq:pqBounds}
\end{align}
in \eqref{eq:pMinCond}--\eqref{eq:qMaxCond}.
While $\Omega_{\rm FI}$ does not map to its interior, i.e.~only $g(\Omega_{\rm FI}) \subset \Omega_{\rm FI}$,
by generalising the approach used in \cite{Si20e} we have observed numerically that $\Omega_{\rm FI}$
can be shrunk by a small amount to produce a trapping region $\Omega \subset \Omega_{\rm FI}$
that necessarily satisfies \eqref{eq:pMinCond}--\eqref{eq:qMaxCond} with the same bounds on $p$ and $q$.

For the collection $\bM_\Omega$, the algorithm in \cite{GlSi21b} also produces the invariant expanding cone
\begin{equation}
C = \left\{ a \begin{bmatrix} \cos(\theta) \\ \sin(\theta) \end{bmatrix} \,\middle|\,
a \in \mathbb{R},\, \theta_0 \le \theta \le \theta_1 \right\},
\nonumber
\end{equation}
where $\theta_0 \approx 0.8062$ and $\theta_1 \approx 2.0227$.
By Proposition 8.1 of \cite{GlSi21b}, this cone can be enlarged slightly to produce
a cone that is contracting-invariant and expanding.

This shows that the border-collision bifurcation of the power converter model
satisfies the conditions of Theorem \ref{th:main}.
We can therefore conclude that the model has a chaotic attractor
for all $\omega_{\rm BCB} < \omega < \omega_{\rm BCB} + \delta$, for some $\delta > 0$.
Based on the numerically computed Lyapunov exponent shown in Fig.~\ref{fig:zm_bifDiag},
we could possibly take $\delta = 0.15$.
By inverting the coordinate changes required to go from the stroboscopic map $p$ to the BCNF $g$,
we were able to reproduce the attractor of $g$ with \eqref{eq:tLdLtRdRValues}
in the coordinates of $p$, and this shown in Fig.~\ref{fig:zm_qqStrobo}(b).

\section{Discussion}
\label{sec:conc}
\setcounter{equation}{0}

The border-collision normal form has attracted a great deal of attention
because it acts as a paradigm for the dynamics of general piecewise-smooth systems,
and because it has a broad variety of applications.
In both contexts it is important to know which features of the dynamics are particular
to the piecewise-linear character of the BCNF,
and which are persistent features of piecewise-smooth systems.
Certainly in applications this question is fundamental to the interpretation of results. 

In this paper we have established techniques that can be used to prove that a
chaotic attractor observed in the BCNF persists with the addition of nonlinear terms close to a border-collision bifurcation.
The techniques are sufficiently simple that the conditions can be checked in explicit examples,
and we have achieved this for a prototypical power converter model.
This approach appears to be quite effective because it is possible
to use analytic (or simple, finite numerical calculations)
to prove the existence of chaotic attractors in the BCNF,
then use persistence arguments to infer the existence of chaotic attractors in the original system.
This eliminates the need to compute asymptotic quantities, such as Lyapunov exponents,
or to rely on a visual examination of numerically computed attractors.

For the power converter model,
the chaotic attractor appears to vary continuously (with respect to Hausdorff metric)
as $\omega$ is varied just past the border-collision bifurcation.
This is why the attractor of the BCNF shown in Fig.~\ref{fig:zm_qqStrobo}(b)
closely resembles that of the power converter shown in Fig.~\ref{fig:zm_qqStrobo}(a).
It remains to determine general conditions that ensure continuity,
perhaps by employing the result of Hoang {\em et.~al.}~\cite{HoOl15}, see \cite{GlSi20b}.

This work feeds into a larger (and often unspoken) question about the BCNF: is it a 
normal form in the strict, bifurcation theory sense of the term \cite{Ku04}?
The dynamics of a (strict) normal form is equivalent to those of the original system
in some neighbourhood of parameter space and phase space.
In contrast, in Theorem \ref{th:main} the neighbourhood $B_{\mu s}(\bO)$ shrinks to a point at the bifurcation.
These concepts depend on the type of equivalence being used and the
result of this paper comes a step closer to showing that, with an appropriate definition of equivalence,
the BCNF can indeed be a normal form in this stronger, technical sense.

\section*{Acknowledgements}

The authors were supported by Marsden Fund contract MAU1809,
managed by Royal Society Te Ap\={a}rangi.

\appendix

\section{Calculations for the regions $E_i$}
\label{app:Ei}
\setcounter{equation}{0}

Here we prove Lemmas \ref{le:EiPhiR} and \ref{le:EiBoundary}.

We first show how the values of $n_q$ and $d_q$ can be computed iteratively.
By substituting $x_2 = n_{q-1} x_1 + d_{q-1}$, which represents $g_R^{-(q-1)}(\Sigma)$, into
\begin{equation}
g_R^{-1}(x) = \begin{bmatrix}
-\frac{1}{\delta_R} \,x_2 \\
x_1 + \frac{\tau_R}{\delta_R} \,x_2 - 1
\end{bmatrix},
\label{eq:gRinverse}
\end{equation}
we obtain
\begin{equation}
g_R^{-1} \left( \begin{bmatrix} x_1 \\ n_{q-1} x_1 + d_{q-1} \end{bmatrix} \right) =
\begin{bmatrix} -\frac{n_{q-1}}{\delta_R} \\ \frac{\tau_R n_{q-1}}{\delta_R} + 1 \end{bmatrix} x_1 +
\begin{bmatrix} -\frac{d_{q-1}}{\delta_R} \\ \frac{\tau_R d_{q-1}}{\delta_R} - 1 \end{bmatrix},
\label{eq:nidiConstruct}
\end{equation}
which represents $g_R^{-q}(\Sigma)$.
From \eqref{eq:nidiConstruct} we see that the slope and $x_2$-intercept of $g_R^{-q}(\Sigma)$ are
\begin{equation}
\begin{split}
n_q &= -\frac{\delta_R}{n_{q-1}} - \tau_R \,, \\
d_q &= -\frac{d_{q-1}}{n_{q-1}} - 1,
\end{split}
\label{eq:nidiRecurrenceRelation}
\end{equation}
assuming $n_{q-1} \ne 0$.
If $n_{q-1} = 0$, then $n_q = \infty$,
and from \eqref{eq:nidiConstruct} we see that $g_R^{-q}(\Sigma)$ is the
vertical line $x_1 = -\frac{d_{q-1}}{\delta_R}$.
Further, by substituting this into \eqref{eq:gRinverse} we obtain
\begin{equation}
g_R^{-1} \left( \begin{bmatrix} -\frac{d_{q-1}}{\delta_R} \\ x_2 \end{bmatrix} \right) =
\begin{bmatrix} -\frac{1}{\delta_R} \\ \frac{\tau_R}{\delta_R} \end{bmatrix} x_2 +
\begin{bmatrix} 0 \\ -\frac{d_{q-1}}{\delta_R} - 1 \end{bmatrix},
\nonumber
\end{equation}
and therefore in this case the slope and $x_2$-intercept of $g_R^{-(q+1)}(\Sigma)$ are
\begin{equation}
\begin{split}
n_{q+1} &= -\tau_R \,, \\
d_{q+1} &= -\frac{d_{q-1}}{\delta_R} - 1.
\end{split}
\label{eq:nidiSpecialCase}
\end{equation}
To obtain starting values for the iterations,
notice $g_R(\Sigma)$ is the $x_1$-axis, so $n_{-1} = d_{-1} = 0$.
By substituting these into \eqref{eq:nidiSpecialCase} we obtain $n_1 = -\tau_R$ and $d_1 = -1$.

In summary, starting with $n_1 = -\tau_R$ and $d_1 = -1$ we can use \eqref{eq:nidiRecurrenceRelation} to
iteratively generate $n_q$ and $d_q$ for all $q \ge 1$,
using instead \eqref{eq:nidiSpecialCase} for the special case $n_{q-1} = 0$.

We now prove Lemmas \ref{le:EiPhiR} and \ref{le:EiBoundary} together.
This is achieved by carefully characterising the regions $E_q$
and here the reader may find it helpful to refer to Fig.~\ref{fig:zm_DpEq}.

\begin{proof}[Proof of Lemmas \ref{le:EiPhiR} and \ref{le:EiBoundary}]
We first describe the backwards orbit of $\bO$ under $g_R$.
Notice $\bO \in \Sigma$ and $g_R^{-1}(\bO) = (0,-1) \in \Sigma$.
Thus, for all $q \ge 0$, the points $g_R^{-q}(\bO)$ and $g_R^{-(q+1)}(\bO)$ lie on $g_R^{-q}(\Sigma)$.
They are distinct points, hence $g_R^{-q}(\Sigma)$ is the unique line that passes through these points.
By \eqref{eq:paramCond2}--\eqref{eq:paramCond3},
\begin{equation}
g_R^{-q}(\bO) \in \{ x \in \mathbb{R}^2 \,\big|\, x_1 > 0,\, x_2 < 0 \}, \qquad
\text{for all $q \ge 2$}.
\label{eq:gRmibO}
\end{equation}
Thus for each $q \ge 2$ and any point $x \in \mathbb{R}^2$ sufficiently close to $g_R^{-q}(\bO)$,
we have $\chi_R(x) \ge q-1$ by the definition of $\chi_R$.
Further, there exist points arbitrarily close to $g_R^{-q}(\bO)$
such that $\chi_R(x) = q-1$.
Therefore
\begin{equation}
\begin{split}
g_R^{-q}(\bO) &\notin {\rm cl}(E_i), \qquad \text{for all $i = 1,\ldots,q-2$}, \\
g_R^{-q}(\bO) &\in \partial E_{q-1} \,,
\end{split}
\label{eq:gRmibO2}
\end{equation}
where ${\rm cl}(\cdot)$ denotes closure and $\partial$ denotes boundary.

Next we characterise the regions $E_q$ up to $q = q^*$.
By definition, $E_1$ consists of all $x \in \Pi_R$ for which $g_R(x)_1 < 0$.
We have $g_R(x)_1 = \tau_R x_1 + x_2 + 1$, therefore
\begin{equation}
E_1 = \left\{ x \in \Pi_R \,\big|\, x_1 \ge 0,\, x_2 < n_1 x_1 + d_1 \right\},
\label{eq:E1}
\end{equation}
recalling $n_1 = -\tau_R$ and $d_1 = -1$.
For each $q \ge 2$, $E_q$ consists of all $x \in \Pi_R$ for which $g_R(x) \in E_{q-1}$.
It follows that
\begin{equation}
E_q = g_R^{-1}(E_{q-1}) \cap \Pi_R \,, \qquad
\text{for all $q \ge 2$.}
\label{eq:Ei2}
\end{equation}
By \eqref{eq:E1}, $E_1$ is the region bounded by two rays emanating from $g_R^{-1}(\bO) = (0,-1)$.
One ray is part of $\Sigma$, the other ray is part of $g_R^{-1}(\Sigma)$ and contains the point $g_R^{-2}(\bO)$.
Then from \eqref{eq:gRmibO2} and \eqref{eq:Ei2}, for all $q \in \{ 1,\ldots, q^* \}$ the region
$E_q$ is bounded by two rays emanating from $g_R^{-q}(\bO)$.
One ray is part of $g_R^{-(q-1)}(\Sigma)$, the other ray is part of $g_R^{-q}(\Sigma)$ and contains the point $g_R^{-(q+1)}(\bO)$.
By \eqref{eq:gRmibO} and the definition of $q^*$,
$E_q \cap \Phi_R = \varnothing$ for all $q \in \{ 1,\ldots, q^*-1 \}$,
whereas $E_{q^*} \cap \Phi_R \ne \varnothing$.

Notice $E_{q^*}$ contains an infinite section of the positive $x_1$-axis.
The preimage of the $x_1$-axis under $g_R$ is the $x_2$-axis,
thus $E_{q^*+1}$ contains an infinite section of the positive $x_2$-axis.
Therefore $E_{q^*+1}$ has three boundaries:
\begin{enumerate}
\setlength{\itemsep}{0pt}
\item 
a ray (part of $g_R^{-q^*}(\Sigma)$ emanating from $g_R^{-(q^*+1)}(\bO)$),
\item
the line segment from $g_R^{-(q^*+1)}(\bO)$ to $\left( 0,d_{q^*+1} \right)$, and
\item
the part of $\Sigma$ above $\left( 0,d_{q^*+1} \right)$.
\end{enumerate}
Finally, if $q^{**} > q^*+1$, then for all $q \in \{ q^*+2,\ldots,q^{**} \}$,
$E_q$ is the triangle with vertices $g_R^{-q}(\bO)$, $(0,d_{q-1})$, and $(0,d_q)$.
This in part relies on the observation $d_{q^{**}} > -1$
which is a consequence of \eqref{eq:nidiRecurrenceRelation} and the definition of $q^{**}$.
For each $q \in \{ q^*+1,\ldots,q^{**} \}$ we have $E_q \cap \Phi_R \ne \varnothing$ because $d_{q-1} > 0$.
Our precise description of the $E_q$ implies
$\Phi_R \subset \bigcup_{q=q^*}^{q^{**}} E_q$, and this completes the proof Lemma \ref{le:EiPhiR}.
Further, we have shown that the boundaries of $E_q$ do not coincide with an interval the $x_1$-axis
or have vertices on the $x_1$-axis, and this completes the proof of Lemma \ref{le:EiBoundary}.
\end{proof}


\end{document}